\documentclass[a4paper,11pt]{amsproc}
\title{Algebraic properties of the ring $C(X)_\mathcal{P}$}
\usepackage{amsfonts, amsmath, amssymb, graphicx, mathrsfs}
\usepackage[all]{xy}
 \usepackage{tikz-cd}
\newdir{ >}{!/6pt/@{ }*:(1,0.2)@_{>}*:(1,-0.2)@^{>}}

\theoremstyle{plain}

\newtheorem{theorem}{Theorem}[section]

\title{Algebraic properties of the ring $C(X)_\mathcal{P}$}
\usepackage{amsfonts, amsmath, amssymb, graphicx, mathrsfs}
\usepackage[all]{xy}

\newdir{ >}{!/6pt/@{ }*:(1,0.2)@_{>}*:(1,-0.2)@^{>}}

\theoremstyle{plain}

\theoremstyle{definition}
\newtheorem{definition}[theorem]{Definition}

\newtheorem{remark}[theorem]{Remark}
\newtheorem{counter example}[theorem]{Counter Example}

\newtheorem{corollary}[theorem]{Corollary}

\newtheorem{example}[theorem]{Example}

\numberwithin{equation}{section}

\author[S. Dey]{Soumajit Dey}	\address{Department of Pure Mathematics, University of Calcutta, 35, Ballygunge Circular Road, Kolkata 700019, West Bengal, India}	\email{deysoumajit8@gmail.com}

\author[S. K. Acharyya]{Sudip Kumar Acharyya}	\address{Department of Pure Mathematics, University of Calcutta, 35, Ballygunge Circular Road, Kolkata
	700019, West Bengal, India}	\email{sdpacharyya@gmail.com}

\author[D. Mandal]{Dhananjoy Mandal} \address{Department of Pure Mathematics, University of Calcutta, 35 , Ballygunge Circular Road, Kolkata 700019, West Bengal, India}  \email{dmandal.cu@gmail.com / dmpm@caluniv.ac.in}


                                                                                          
\begin{document}
	
	\title [Structure spaces and allied problems on a class of  $\mathcal{M}(X,\mathcal{A})$ ]{Structure spaces and allied problems on a class of rings of measurable functions }

	\thanks {The first author extends immense gratitude and thanks to the University Grants Commission, New Delhi, for the award of research fellowship (NTA Ref. No. 211610214962).}

	\begin{abstract}
		A ring $S(X,\mathcal{A})$ of real valued $\mathcal{A}$-measurable functions defined over a measurable space $(X,\mathcal{A})$ is called a $\chi$-ring if for each $E\in \mathcal{A}  $, the characteristic function $\chi_{E}\in S(X,\mathcal{A})$. The set $\mathcal{U}_X$ of all $\mathcal{A}$-ultrafilters on $X$ with the Stone topology $\tau$ is seen to be homeomorphic to an appropriate quotient space of the set $\mathcal{M}_X$ of all maximal ideals in $S(X,\mathcal{A})$ equipped with the hull-kernel topology $\tau_S$. It is realized that $(\mathcal{U}_X,\tau)$ is homeomorphic to $(\mathcal{M}_S,\tau_S)$ if and only if  $S(X,\mathcal{A})$ is a Gelfand ring. It is further observed that $S(X,\mathcal{A})$ is a Von-Neumann regular ring if and only if each ideal in this ring is a $\mathcal{Z}_S$-ideal and $S(X,\mathcal{A})$ is Gelfand when and only when every maximal ideal in it is a $\mathcal{Z}_S$-ideal. A pair of topologies $u_\mu$-topology and $m_\mu$-topology,  are introduced on the set $S(X,\mathcal{A})$ and a few properties are studied. 
			\end{abstract}	
	\maketitle

	\section{Introduction}
	Our beginning is with a pair of objects $(X,\mathcal{A})$, here $X$ is a non empty set and $\mathcal{A}$, a $\sigma$-algebra of subsets of $X$. Such a pair is often called a measurable space. A function $f:X\rightarrow \mathbb{R}$ is called $\mathcal{A}$-measurable or simply measurable, when there is no chance of any confusion if for any open set $U$ in $\mathbb{R}$, $f^{-1}(U) $ is a measurable set meaning that  $f^{-1}(U)\in\mathcal{A}$. The collection $\mathcal{M}(X,\mathcal{A})$ of all real valued measurable functions on $X$ makes a commutative lattice ordered ring with unity, if the relevant operations are defined pointwise on $X$. The subcollection $\mathcal{M}^*(X,\mathcal{A})$ containing those functions in $\mathcal{M}(X,\mathcal{A})$ which are bounded over $X$, is a subring as well as a sublattice of $\mathcal{M}(X,\mathcal{A})$. Some pertinent problems connected with these two rings have already been addressed by several experts in this area. One can consult the articles \cite{AARD2021}, \cite{AABS2020}, \cite {ABS2020}, \cite{AHM2009}, \cite{EMY2018}, \cite{G1966}, \cite{M2010}, in this connection. We call any ring lying between $\mathcal{M}^*(X,\mathcal{A})$ and $\mathcal{M}(X,\mathcal{A})$ an intermediate ring of measurable functions. As far as we dig into the literature there is only one paper \cite{ABS2020} concerning a few relevant problems on the intermediate rings of measurable functions. We let $\sum(X,\mathcal{A})$ denote the family of all such intermediate rings. It is established in \cite{ABS2020}, amongst others that all the rings in the family $\sum(X,\mathcal{A})$ have identical structure spaces. The structure space of a commutative ring with unity stands for the set of all maximal ideals in the ring, equipped with the hull-kernel topology. Being motivated by this interesting result, in the present paper we endeavor to determine the structure spaces of a kind of subrings of $\mathcal{M}(X,\mathcal{A})$, which we designate as $\chi$-rings and the aggregate of such rings in general, contain the family $\sum(X,\mathcal{A})$ of measurable functions properly. Indeed we call any subring of $\mathcal{M}(X,\mathcal{A})$ which contains the characteristic functions of all measurable sets a $\chi$-ring. It is clear that any intermediate rings of measurable functions is a $\chi$-ring. The aggregate of all simple Lebesgue measurable functions over the closed interval $[0,1]$ is a $\chi$-ring which is not an intermediate ring of Lebesgue measurable functions. One of the important features of a large class of function rings, including the rings of real valued or complex valued continuous functions, intermediate rings of real valued or complex valued continuous functions defined over Tychonoff spaces and intermediate rings of measurable functions defined over measurable spaces, is that each such ring is a Gelfand ring in the sense that every prime ideal in any such ring extends to a unique maximal ideal. This particularly important fact enjoyed by these function rings facilitates a lot towards determining the structure spaces of these function rings, see \cite{ABR2021}, \cite{ARN2022}, \cite{AP2016}, \cite{AABJ2021}. In contrast we show in the present article that a $\chi$-ring need not be Gelfand. Incidentally we find out several necessary and sufficient conditions for a $\chi$-ring $S(X,\mathcal{A})$ to be Gelfand. The nature of these conditions vary from being set theoretic to the extent  of being algebraic and also topological. To attain these conditions, we associate with each proper ideal $I$ in $S(X,\mathcal{A})$, a filter $\mathcal{Z}_S[I]$ of measurable subsets of $X$ described as follows: $\mathcal{Z}_S[I]$ stands for the family of those measurable sets $E\in \mathcal{A} $ for which some function $f\in I$ is invertible in $X\setminus E$, meaning that there exists $g\in S(X,\mathcal{A})$ such that $f(x)g(x)=1$ for all $x\in X\setminus E$. If $I$ is a pseudo prime ideal in $S(X,\mathcal{A})$, in particular a maximal ideal in $S(X,\mathcal{A})$, then it turns out $\mathcal{Z}_S[I]$ is an $\mathcal{A}$-ultrafilter on $X$ i.e., a maximal filter of $\mathcal{A}$-measurable sets. On the other hand it is proved that if $\mathcal{U}$ is an $\mathcal{A}$-ultrafilter on $X$, then there is a maximal ideal $M$ in $S(X,\mathcal{A})$ such that $\mathcal{Z}_S[M]=\mathcal{U}$. Thus the map $\mathcal{Z}_S:\begin{cases}
		\mathcal{M}_S\rightarrow \mathcal{U}(X)\\
		M \rightarrow \mathcal{Z}_S[M]
	\end{cases}$ defined on the set $\mathcal{M}_S$ of all maximal ideals in $S(X,\mathcal{A})$ into the set $\mathcal{U}(X)$ of all $\mathcal{A}$-ultrafilters on $X$ is essentially onto $\mathcal{U}(X)$. In general this last map is not one-to-one. We furnish a counter example in this paper to support this statement. However we establish that the map $\mathcal{Z}_S:\mathcal{M}_S\rightarrow \mathcal{U}(X)$ is one-to-one if and only if $S(X,\mathcal{A})$ is a Gelfand ring. If the $\chi$-ring $S(X,\mathcal{A})$ is in particular an intermediate ring of measurable functions defined over the measurable space $(X,\mathcal{A})$, then the map $\mathcal{Z}_S:\mathcal{M}_S\rightarrow \mathcal{U}(X)$  becomes a homeomorphism if $\mathcal{M}_S$ is equipped with the hull-kernel topology $\tau_S$ and the Stone topology $\tau$ is imposed on $\mathcal{U}(X)$ [see Theorem $4.8$ in \cite{ABS2020}]. Things are not that pleasant here for an arbitrary $\chi$-ring $S(X,\mathcal{A})$. In fact a suitable quotient space of $(\mathcal{M}_S,\tau_S)$ is seen to be homeomorphic to the space $(\mathcal{U}(X),\tau)$. A little elaboration may help to realize this situation. On defining an equivalence relation $\sim$ on $\mathcal{M}_S$ by writing $M\sim N$, $M,N \in \mathcal{M}_S$ if and only if $\mathcal{Z}_S[M]=\mathcal{Z}_S[N]$ and then defining the quotient topology $\tau_q$ on the set $\mathcal{M}_S/\sim$ induced by the natural map $\Pi:\begin{cases}
	\mathcal{M}_S\rightarrow \mathcal{M}_S/\sim \\
	M\rightarrow [M]
	
\end{cases}$ where $[M]=\{N\in \mathcal{M}_S:M\sim N\}$, it is proved in the present article that the space $(\mathcal{M}_S/\sim,\tau_q)$ is homeomorphic to the space $(\mathcal{U}(X),\tau)$. It is deduced from this result that the map $\mathcal{Z}_S:\begin{cases}
(\mathcal{M}_S,\tau_S)\rightarrow(\mathcal{U}(X),\tau)\\ 
M\rightarrow \mathcal{Z}_S[M]
\end{cases}$ is a homeomorphism if and only if $S(X,\mathcal{A})$ is a Gelfand ring. Since the structure space of any intermediate ring $N(X,\mathcal{A})$ of measurable functions defined over $(X,\mathcal{A})$ is already homeomorphic to the space $(\mathcal{U}(X),\tau)$, under the map $\mathcal{Z}_S$, as mentioned above, it follows that each such ring $N(X,\mathcal{A})$ is a Gelfand ring, a fact possibly missed by the authors in \cite{ABS2020}. In this context we would like to mention that a topology $\tau_\chi$ on $\mathcal{M}_S$, defined via the characteristic functions of measurable sets in $X$ and coarser than the hull-kernel topology $\tau_S$ on $\mathcal{M}_S$ is introduced in an appropriate place. It is realized that the map $\mathcal{Z}_S^\chi:\begin{cases}
(\mathcal{M}_S,\tau_\chi)\rightarrow (\mathcal{U}(X),\tau)\\M\rightarrow\mathcal{Z}_S[M]
\end{cases}$ is a homeomorphism if and only if $S(X,\mathcal{A})$ is a Gelfand ring. Thus altogether four topologies viz $\tau_S,\tau_\chi$ on $\mathcal{M}_S$, $\tau_q$ on $\mathcal{M}_S/\sim$ and $\tau$ on $\mathcal{U}(X)$ are introduced in course of determining the structure space of the $\chi$-ring $S(X,\mathcal{A})$. It turns out that if $S(X,\mathcal{A})$ is a Gelfand ring, then  all these four topologies are pairwise homoemorphic.
\newline The results as obtained in this article and deliberated as above constitute a bulk of Section 2 and Section 3 of the paper. It is well known that the ring $\mathcal{M}(X,\mathcal{A})$ of all real valued measurable functions defined over a measurable space $(X,\mathcal{A})$ is Von-Neumann regular \cite{ABS2020}, \cite{AHM2009}. In contrast in this present article we realized that a $\chi$-ring $S(X,\mathcal{A})$ is not necessarily Von-Neumann regular. However in the Section $4$ of this article we find out several conditions each necessary and sufficient for the Von-Neumann regularity of $S(X,\mathcal{A})$ and one such typical condition is : if $f,g\in S(X,\mathcal{A})$ and $Z(f)\subseteq Z(g)$, then $g$ is a multiple of $f$ in $S(X,\mathcal{A})$, here $Z(f)=\{x\in X:f(x)=0\}$ is the zero set of the function $f$. It follows from the equivalence of this condition with Von-Neumann regularity of $S(X,\mathcal{A})$ that, no intermediate ring lying strictly between $\mathcal{M}^*(X,\mathcal{A})$ and $\mathcal{M}(X,\mathcal{A})$ can ever be Von-Neumann regular. This fact is proved independently in \cite{ABS2020}. Incidentally a second typical condition in this list of necessary and sufficient conditions for Von-Neumann regularity of $S(X,\mathcal{A})$ says that: $S(X,\mathcal{A})$ is Von-Neumann regular if and only if each ideal $I$ in $S(X,\mathcal{A})$ is a $\mathcal{Z}_S$-ideal. $I$ is called a $\mathcal{Z}_S$-ideal in $S(X,\mathcal{A})$ if whenever $\mathcal{Z}_S(f)\subseteq \mathcal{Z}_S[I]$, $f\in S(X,\mathcal{A})$, the conclusion is that $f\in I$, here $\mathcal{Z}_S(f)=\{E\in \mathcal{A}:f \text{ is invertible in } X\setminus E \text{ in the ring } S(X,\mathcal{A})\}$ and we recall that $\mathcal{Z}_S[I]=\bigcup\limits_{g\in I} \mathcal{Z}_S(g) $. We would like to mention at this point that the Gelfandness of this ring $S(X,\mathcal{A})$ can be captured via this $\mathcal{Z}_S$-ideal in the following manner: $S(X,\mathcal{A})$ is Gelfand if and only if each maximal ideal in $S(X,\mathcal{A})$ is a $\mathcal{Z}_S$-ideal.
\newline  In the final and Section 5 of this article we initiate two topologies viz the $u_\mu$-topology and the $m_\mu$-topology on the set $S(X,\mathcal{A})$ in the presence of a measure $\mu$ on the $\sigma$-algebra $\mathcal{A}$ on $X$. A map $\mu:\mathcal{A}\rightarrow[0,\infty]$ is called a measure if $\mu(\emptyset)=0$ and $\mu$ is countably additive in the sense that if $\{E_n\}_{n=1}^\infty$ is a countable family of pairwise disjoint members of $\mathcal{A}$, then $\mu(\bigcup\limits_{n=1}^\infty E_n)=\sum\limits_{n=1}^\infty \mu(E_n)$. 

Let $\mathcal{A}_0=\{A\in \mathcal{A}:\mu(A)=0\}$. 
In this section $(5)$ we make an additional assumption that each constant function $r , r\in \mathbb{R}$ is a member of $S(X,\mathcal{A})$. Here $r$ stands for the function $r(x)=r$ for each $x\in X$. It is easy to observe that each intermediate ring of measurable functions satisfies this additional condition. We introduce the following notations which we will use throughout the Section 5 of this article. Let $f\in S(X,\mathcal{A})$, $u\in \mathcal{M}(X,\mathcal{A})$ with $u(x)>0$ for each $x\in X$, $\epsilon>0$ in $\mathbb{R}$  and  $E_0\in \mathcal{A}_0$. We set $u(f,\epsilon, E_0)=\{g\in S(X,\mathcal{A}):\sup\limits_{x\in X\setminus E_0}|f(x)-g(x)|<\epsilon\}$ and $m(f,u,E_0)=\{g\in S(X,\mathcal{A}):|f(x)-g(x)|<u(x)\text{ for each } x\in X\setminus E_0\}$. Then the family $\{u(f,\epsilon,E_0):f\in S(X,\mathcal{A}),\epsilon>0 \text{ in }\mathbb{R}, E_0\in \mathcal{A}_0\}$ is an open base for a topology which we call the $u_\mu$-topology on the set $S(X,\mathcal{A})$. It is further realized that the family $\{m(f,u,E_0):f\in S(X,\mathcal{A}), u\in \mathcal{M}(X,\mathcal{A}), u(x)>0 \text{ for each }x\in X\text{ and } E_0\in \mathcal{A}_0\}$ too is an open base for a topology finer than the $u_\mu$-topology and which we designated as the $m_\mu$-topology on $S(X,\mathcal{A})$. We show that $S(X,\mathcal{A})$ with the $u_\mu$-topology is a pseudometrizable topological group while $S(X,\mathcal{A})$ with the $m_\mu$-topology is a topological ring. The component of the function $0$ in  $S(X,\mathcal{A})$ in the $u_\mu$-topology is found out to be $L^\infty(X,\mathcal{A},\mu)\cap S(X,\mathcal{A})$, here $L^\infty(X,\mathcal{A},\mu)$ stands for the set of all those functions in $\mathcal{M}(X,\mathcal{A})$, which are essentially bounded ( with respect to the measure $\mu$ ) on $X$. It follows from this result that a $\chi$-ring $S(X,\mathcal{A})$ is connected in the $u_\mu$-topology if and only if it is contained in $L^\infty(X,\mathcal{A},\mu)$. We find out several conditions each necessary and sufficient for the coincidence of the two topologies viz the $u_\mu$-topology and the $m_\mu$-topology on $S(X,\mathcal{A})$. One such condition says that $u_\mu$-topology = $m_\mu$-topology on $S(X,\mathcal{A})$ if and only if the $m_\mu$-topology is first countable.
 	\section{Ideals in $\chi$-rings versus filters of measurable sets.}
	\begin{definition}
		We call a ring $S(X,\mathcal{A})$ a $\chi$-ring if it is a subring of $\mathcal{M}(X,\mathcal{A})$ and if in addition for any $A\in \mathcal{A}$, the characteristic function $\chi_A \in S(X,\mathcal{A})$. It follows immediately that the ring $\mathcal{M}^*(X,\mathcal{A})$ is a $\chi$-ring, in particular all the intermediate rings of measurable functions lying between $\mathcal{M}^*(X,\mathcal{A})$ and $\mathcal{M}(X,\mathcal{A})$ are $\chi$-rings. An ideal $I$ unmodified in the ring $S(X,\mathcal{A})$ always stands for a proper ideal, i.e., for which $I\neq S(X,\mathcal{A})$.
	\end{definition} It is trivial that $\mathcal{M}(X,\mathcal{A})$ is the largest $\chi$-ring. On the other extreme the ring $\mathcal{M}_F(X,\mathcal{A})=\{f\in \mathcal{M}(X,\mathcal{A}):f(X)\text{ is a finite subset of }\mathbb{Z}\}$, is the smallest $\chi$-ring as can be verified with out any difficulty. If $X$ is an infinite set and the $\sigma$-algebra $\mathcal{A}$ on $X$ is infinite then $\mathcal{M}_F(X,\mathcal{A})$ is a $\chi$-ring which is not intermediate rings. Less elementary example of $\chi$-rings which are not intermediate rings can be constructed as follows: \newline
 If $I $ is an ideal in the ring $\mathcal{M}(X,\mathcal{A})$, then $\mathcal{M}_F(X,\mathcal{A})+I$ is a $\chi$-ring which is not an intermediate ring. To show that this later class of rings are not an intermediate ring, it suffices to show that $\mathcal{M}_F(X,\mathcal{A})+I \neq \mathcal{M}^*(X,\mathcal{A})$. Indeed the constant function $\frac{1}{2}\notin \mathcal{M}_F(X,\mathcal{A})+I$ for if it be so, then there exists $f \in \mathcal{M}_F(X,\mathcal{A})$ such that $\frac{1}{2}=f+g$ for some $g\in I$. But then $g=\frac{1}{2}-f$ does not vanish any where on $X$, in other words $g$ is a unit in $\mathcal{M}(X,\mathcal{A})$,  a contradiction to the choice that $I$ is a proper ideal in $\mathcal{M}(X,\mathcal{A})$.\newline If the $\sigma$-algebra $\mathcal{A}$ on $X$ is infinite then the set of all simple functions over $(X,\mathcal{A})$ is a $\chi$-ring and is a proper subring of $\mathcal{M}^*(X,\mathcal{A}) $ and is therefore not an intermediate ring.
 \newline In what follows $S(X,\mathcal{A})$ always stands for a $\chi$-ring. Borrowing the terminologies of local invertibility of real valued continuous functions from  \cite{BW1991}, \cite{RW1987}, \cite{RW1997}, we introduce below the appropriate measure theoretic analogue of local invertibility of functions in the $\chi$-ring $S(X,\mathcal{A})$.
 \newline \begin{definition}
 	 An $f\in S(X,\mathcal{A})$ is called $E$-regular , where $E\in \mathcal{A}$, if there exists a $g\in S(X,\mathcal{A})$ such that $fg|_E=1$.
 	 
 \end{definition}
	
For $f\in S(X,\mathcal{A})$, we set $\mathcal{Z}_S(f)=\{E\in \mathcal{A}: f \text{ is }E^c \text{-regular }\}$, here $E^c=X\setminus E$. We set further for any ideal $I$ in $S(X,\mathcal{A})$, $\mathcal{Z}_S[I]=\bigcup\limits_{f\in I}\mathcal{Z}_S(f)$ and for any filter $\mathcal{F}$ of  $\mathcal{A}$-measurable sets in $X$, $\mathcal{Z}_\mathcal{M} ^{-1}[\mathcal{F}]=\{g\in \mathcal{M}(X,\mathcal{A}): \mathcal{Z}_\mathcal{M}(g ) \subseteq \mathcal{F} \}$, here $\mathcal{Z}_\mathcal{M}(g)=\{E\in \mathcal{A}:\text{ there exists } h\in \mathcal{M}(X,\mathcal{A}) \text{ such that } gh|_{E^c}=1\}
$. It is easy to prove by using the techniques adopted in \cite{BW1991}, \cite{RW1987} that $\mathcal{Z}_S(f)$ is an $\mathcal{A}$-filter on $X$ if and only if $f$ is not a unit in $S(X,\mathcal{A})$ and $\mathcal{Z}_S[I]$ is an $\mathcal{A}$-filter on $X$. Furthermore it is easy to check that if $f$ is a non unit in $\mathcal{M}(X,\mathcal{A})$, then $\mathcal{Z}_\mathcal{M}(f)=\{E\in \mathcal{A}:Z(f) \subseteq E\}$. The following theorem shows that each $\mathcal{A}$-filter on $X$ is the $\mathcal{Z}_S$ image of an ideal in $S(X,\mathcal{A})$. \begin{theorem}\label{ preimage of filter}
	Let $\mathcal{F}$ be an $\mathcal{A}$-filter on $X$. Then $\mathcal{Z}_S[\mathcal{Z}_\mathcal{M}^{-1}(\mathcal{F}) \cap S(X,\mathcal{A})]=\mathcal{F}$
\end{theorem}
\begin{proof}
	It follows from the Theorem $4.3(ii)$ in \cite{ABS2020} that $\mathcal{Z}_\mathcal{M}^{-1}(\mathcal{F})$ is an ideal in $\mathcal{M}(X,\mathcal{A})$ and hence $\mathcal{Z}_\mathcal{M}^{-1}(\mathcal{F}) \cap S(
	X,\mathcal{A})$ is an ideal in $S(X,\mathcal{A})$. For any $E\in \mathcal{F}$, $\chi_{E^c}$ is invertible on $E^c$ in the ring $S(
	X,\mathcal{A})$. Also since $Z(\chi_{E^c})=E$, it follows that $E\in \mathcal{Z}_\mathcal{M}(\chi_{E^c})$, consequently $Z(\chi_{E^c})\in \mathcal{F}$ and hence  $\chi_{E^c} \in \mathcal{Z}_\mathcal{M}^{-1}(\mathcal{F})$. Thus $\chi_{E^c}\in \mathcal{Z}_\mathcal{M}^{-1}(\mathcal{F}) \cap S(X,\mathcal{A})$ and therefore $E\in \mathcal{Z}_S(\chi_{E^c})$. So $E\in \mathcal{Z}_S[\mathcal{Z}_\mathcal{M}^{-1}(\mathcal{F})\cap S(X,\mathcal{A})]$. Thus $\mathcal{F} \subseteq \mathcal{Z}_S[\mathcal{Z}_\mathcal{M}^{-1}(\mathcal{F}) \cap S(X,\mathcal{A})]$.\newline To prove the reverse implication relation choose $f\in \mathcal{Z}_\mathcal{M}^{-1}(\mathcal{F}) \cap S(X,\mathcal{A})$, then any $\mathcal{A}$-measurable set which contains $Z(f)$ is a member of $\mathcal{F}$. Therefore if $E\in \mathcal{Z}_S(f)$, then $Z(f) \subseteq E$ and hence $E\in \mathcal{F}$. Thus $\mathcal{Z}_S(f)\subseteq \mathcal{F}$ and consequently $\mathcal{Z}_S[\mathcal{Z}_\mathcal{M}(\mathcal{F})\cap S(X,\mathcal{A})]\subseteq \mathcal{F}$.
\end{proof}
\begin{remark} On using the notations $\mathcal{I}_S$ and $\mathcal{A}_S$ respectively for the family of all ideals in the ring $S(X,\mathcal{A})$ and the collection of all $\mathcal{A}$-filters on $X$, the last theorem together with the observations that precede this theorem can lead to the following fact: The map $\mathcal{Z}_S:\begin{cases}
		\mathcal{I}_S\rightarrow \mathcal{A}_X\\
		I\rightarrow \mathcal{Z}_S[I]\\
	\end{cases}$ is a surjection onto $\mathcal{A}_X$. We shall soon observe that this map carries maximal ideals in $S(
X,\mathcal{A})$ onto $\mathcal{A}$-ultrafilters on $X$. Indeed the first among the next two theorems says that we can speak something more.\end{remark} \begin{theorem}
\label{pseudoprime}	 Let $P$ be a pseudo prime ideal in $S(X,\mathcal{A})$, in the sense that whenever $f\cdot g=0$, $f,g\in S(X,\mathcal{A})$, then either $f\in P$ or $g\in P$. Then $\mathcal{Z}_S[P]$ ia an $\mathcal{A}$-ultrafilter on $X$.\end{theorem}
 \begin{proof}
 	Since each $\mathcal{A}$-filter on $X$, extends to an $\mathcal{A}$-ultrafilter (by the typical Zorn`s lemma argument), there exists an $\mathcal{A}$-ultrafilter $\mathcal{U}$  on $X$ such that $\mathcal{Z}_S[P] \subset \mathcal{U}$. It is enough to show that $\mathcal{Z}_S[P]=\mathcal{U}$. So let $E\in \mathcal{U}$. Then $\chi_E \cdot \chi_{(X-E)}=0$, implies due to the pseudo primeness of $P$ that $\chi_E\in P$ or $\chi_{(X-E)}\in P$. We assert that $\chi_{(X-E)}\in P$. If possible, let $\chi_{(X-E)}\notin P$, then $\chi_E\in P$. Since $Z(\chi_{E})=X\setminus E$, this implies that $X\setminus E \in \mathcal{Z}_S(\chi_E)$ and consequently $X\setminus E \in \mathcal{Z}_S[P]$ and hence $X\setminus E \in \mathcal{U}$. Since $E$ is already a member of $\mathcal{U}$ and $E\cap (X\setminus E)= \emptyset$, we arrive at a contradiction. Hence $\chi_{(X-E)}\in P$. Consequently $E\in \mathcal{Z}(\chi_{(X-E)})\subseteq \mathcal{Z}_S[P]$ and hence $\mathcal{U} \subseteq \mathcal{Z}_S[P]$.
 \end{proof}
\begin{remark}
	\label{maxi to ultra}
	If $M$ is a maximal ideal in $S(X,\mathcal{A})$ containing a pseudo prime ideal $P$, then $\mathcal{Z}_S[P]=\mathcal{Z}_S[M]=$an $ \mathcal{A}$-ultrafilter on $X$.
\end{remark}
\begin{theorem}\label{preimage of ultra}
	Let $\mathcal{U}$ be an $\mathcal{A}$-ultrafilter on $X$. Then there exists a maximal ideal $M$ in $S(X,\mathcal{A})$ such that $\mathcal{Z}_S[M]=\mathcal{U}$.
\end{theorem}
\begin{proof}
	It follows from Theorem $4.5(ii)$ in \cite{ABS2020} that $\mathcal{Z}_\mathcal{M}^{-1}[\mathcal{U}]$ is a maximal ideal in the ring $\mathcal{M}(X,\mathcal{A})$ and thus $\mathcal{Z}_\mathcal{M}^{-1}[\mathcal{U}]\cap S(X,\mathcal{A})$ is a prime ideal in the ring $S(X,\mathcal{A})$. Hence we get from the Theorem  \ref{pseudoprime} that $\mathcal{Z}_S[\mathcal{Z}_\mathcal{M}^{-1}[\mathcal{U}] \cap S(X,\mathcal{A})]=\mathcal{U}$. Now $\mathcal{Z}_\mathcal{M}^{-1}[\mathcal{U}] \cap S(X,\mathcal{A})$ can be extended to a maximal ideal  $M$ in $S(X,\mathcal{A})$. It is clear that $\mathcal{Z}_S[M]=\mathcal{U}$.
\end{proof}
\begin{remark}
	 The map $\mathcal{Z}_S:\begin{cases}
	 	\mathcal{M}_S\rightarrow \mathcal{U}(X)\\
	 	M\rightarrow \mathcal{Z}_S[M]\\
	 \end{cases}$ takes the maximal ideals in $S(X,\mathcal{A})$ onto the $\mathcal{A}$-ultrafilters on $X$. Here $\mathcal{M}_S$ is the collection of all maximal ideals in $S(X,\mathcal{A})$ and $\mathcal{U}(X)$ stands for the set of all $\mathcal{A}$-ultrafilters on $X$. The map $\mathcal{Z}_S:\mathcal{M}_S\rightarrow \mathcal{U}(X)$ need not be one-to-one. This can be easily realized by looking at the example \ref{counter example}. The next theorem decides when does the $\mathcal{Z}_S:\mathcal{M}_S\rightarrow \mathcal{U}(X)$ become one-to-one.
 
\end{remark}
\begin{theorem}\label{bijection}
	The map $\mathcal{Z}_S:\begin{cases}
		\mathcal{M}_S\rightarrow \mathcal{U}(X)\\
		M\rightarrow \mathcal{Z}_S[M]\\
	\end{cases}$ is one-to-one if and only if $S(X,\mathcal{A})$ is a Gelfand ring.
\end{theorem}
	\begin{proof}
		First assume that $\mathcal{Z}_S$ is a one-to-one map. Let $P$ be a prime ideal in $S(X,\mathcal{A})$ and $M$ and $N$ are two maximal ideals in the same ring such that $P\subseteq M$ and $P \subseteq N$. It follows from the Remark \ref{maxi to ultra} that $\mathcal{Z}_S[P]=\mathcal{Z}_S[M]=\mathcal{Z}_S[N]$. Since $\mathcal{Z}_S$ is assumed to be one-to-one, this implies that $M=N$. Thus $S(X,\mathcal{A})$ is a Gelfand ring.
	\newline To prove the  other part of the theorem let $S(X,\mathcal{A})$ be a Gelfand ring. Let $M$ and $N$ be two maximal ideals in the ring $S(X,\mathcal{A})$ such that $\mathcal{Z}_S[M]=\mathcal{Z}_S[N]$. We shall show that $M=N$. It follows from the Theorem \ref{pseudoprime}	that $\mathcal{Z}_S[M]=\mathcal{Z}_S[N]=\mathcal{U}$, an $\mathcal{A}$-ultrafilter on $X$. We observe in course of previous Theorem \ref{pseudoprime} that $\mathcal{Z}_\mathcal{M}^{-1}[\mathcal{U}]\cap S(X,\mathcal{A})$ is a prime ideal in $S(X,\mathcal{A})$. We shall prove that $  \mathcal{Z}_\mathcal{M}^{-1}[\mathcal{U}] \cap S(X,\mathcal{A}) \subseteq M\cap N$.........(1); from this it will follow, because $S(X,\mathcal{A})$ is a Gelfand ring that $M=N$. Proof of the implication relation (1): let $f\in \mathcal{Z}_\mathcal{M}^{-1}[\mathcal{U}]\cap S(X,\mathcal{A})$. Then $\mathcal{Z}_\mathcal{M}(f) \subseteq \mathcal{Z}_S[M]$, consequently there exists a $g\in M$ such that $\mathcal{Z}_\mathcal{M}(f)\subseteq \mathcal{Z}_S(g)$. As observed prior to the statement of Theorem \ref{ preimage of filter} that $\mathcal{Z}_\mathcal{M}(f)$ contains precisely those $\mathcal{A}$-measurable sets which contains $Z(f)$, from this it follows that $Z(f)\in \mathcal{Z}_S(g)$. This means that $g $ is $X\setminus Z(f)$ invertible in the ring $S(X,\mathcal{A})$. Hence there exists $h\in S(X,\mathcal{A})$ such that $(gh)|_{(X \setminus Z(f))}=1$. Consequently $f=ghf$ on the whole of $X$. Since $g\in M$ and $fh\in S(X,\mathcal{A})$, it follows that $f\in M$. It can be proved analogously that $f\in N$. Thus $f\in M\cap N $. Therefore the implication relation (1) above is proved.
\end{proof} 
\begin{corollary}\label{intermediate corollary}
	Each intermediate ring of measurable functions lying between $\mathcal{M}^*(X,\mathcal{A})$ and $\mathcal{M}(X,\mathcal{A})$ is a Gelfand ring.
	\\ This follows from Theorem $4.6$ in \cite{ABS2020} and the Theorem \ref{bijection} above.\\
	
	 In the following we construct a $\chi$-ring which is not Gelfand.
\end{corollary}
\begin{example}\label{counter example}
	Consider the $\chi$-ring $\mathcal{M}_F(\mathbb{R},P(\mathbb{R}))=\{f:\mathbb{R}\rightarrow \mathbb{R}:f $ is a function with $f(\mathbb{R}) \text{ a finite subset of }\mathbb{Z}\}$. Let $P=\{f\in\mathcal{M}_F(\mathbb{R},P(\mathbb{R})):f(2)=0\}$. Then it is easy to check that $P$ is a prime ideal in the ring under consideration. Clearly $P \subset <P,2>\equiv$ the ideal in this ring generated by $P$ and the constant function $2$ and also $P\subset <P,3>$. It is clear that $<P,2> \neq <P,3>$. We shall check that $<P,2>$ is a maximal ideal in $\mathcal{M}_F(\mathbb{R},P(\mathbb{R}))$. (The proof for the maximality of $<P,3>$ as an ideal is analogous). For that purpose choose $f\in \mathcal{M}_F(\mathbb{R},P(\mathbb{R})) \setminus <P,2>$. We assert that $f(2)$ is an odd integer, for if $f(2)=2m$ for some $m\in \mathbb{Z}$, then $f-2m$ vanishes at the point $2$ and therefore $f-2m\in P$, which implies that $f\in <P,2>$, a contradiction. So we can write $f(2)=2n+1$ for some $n\in \mathbb{Z}$. Now we define the three functions $k,$ $l,$ $h\in \mathcal{M}_F(\mathbb{R},P(\mathbb{R}))$ as follows: $k=\chi_{\{2\}}$, $l(x)=\begin{cases}
		2 \text{ if } x\neq 2 \\
		-n \text{ if } x=2
	\end{cases}$ and $h(x)= \begin{cases}
0 \text{ if } x=2\\
-3 \text{ if } x\neq 2
\end{cases}$. Then it is easy to check that $1=k\cdot f+h+2l$. We note that $h\in P$ and therefore $1\in <f>+<P,2>.$ Since $f$ is chosen arbitrarily outside  $<P,2>$ in the ring $\mathcal{M}_F(\mathbb{R},P(\mathbb{R}))$, it follows that $<P,2> $ is a maximal ideal in this ring.
\end{example}

\section{ Structure space of the $\chi$-ring $S(X,\mathcal{A})$ vis-a-vis the space of ultrafilters of $\mathcal{A}$-measurable sets }

	For any $f\in S(X,\mathcal{A})$, we write $\mathcal{M}_f=\{M\in \mathcal{M}_S:f\in M\}$, here we recall that $\mathcal{M}_S$ stands for the set of all maximal ideals in $S(X,\mathcal{A})$. Then like any commutative ring with unity, the family $\{M_f:f\in S(X,\mathcal{A})\}$ is a base for closed sets of the hull-kernel topology on $\mathcal{M}_S$ [ see $7M$ in \cite{GJ}]. We wish to denote this topology by the symbol $\tau_S$. It follows 
	from [ $7M$ \cite{GJ}] 
	that $(\mathcal{M}_S,\tau_S)$ is a compact $T_1$ topological space. On the other hand since $\mathcal{M}_{\chi_{X}}=\emptyset$ and for any two $\mathcal{A}$-measurable sets $E,F$, $\mathcal{M}_{\chi_{E}}\bigcup \mathcal{M}_{\chi_{F}}=\mathcal{M}_{\chi_{(E\cap F)}}$, as can be easily verified, it follows that the family $\{\mathcal{M}_{\chi_{E}}:E \in \mathcal{A}\}$ is a base for the closed sets of some topology $\tau_\chi$ on $\mathcal{M}_S$. It is clear that $\tau_\chi \subseteq \tau_S$. \newline On writing for any $\mathcal{A}$-measurable set $E$ in $X$, $S_E=\{\mathcal{U}\in \mathcal{U}_X:E\in \mathcal{U}\}$, where we recall that $\mathcal{U}_X$ is the aggregate of all ultrafilters of $\mathcal{A}$-measurable sets in $X$, it is well known that the family $\{S_E:E\in \mathcal{A}\}$ is a base for the closed sets for the Stone-topology $\tau$ on $\mathcal{U}_X$. See the Theorems $2.9, 2.10 $ in \cite{ABS2020}, for a detailed proof of a few results related to this Stone topology. It follows from Theorem $2.10$ \cite{ABS2020} that $(\mathcal{U}_X,\tau)$ is a compact Hausdorff zero-dimensional space. In this section we shall relate these three topologies.\begin{theorem}\label{Z chi continous}
		 The map $\mathcal{Z}_S^\chi:\begin{cases}
		 	(\mathcal{M}_S,\tau_\chi)\rightarrow (\mathcal{U}_X,\tau)\\
		 	M\rightarrow \mathcal{Z}_S[M]\\
		 \end{cases}$ is continuous.
	\end{theorem}
\begin{proof}
	 We shall show that $\mathcal{Z}_S^\chi$ returns back a basic closed set in the range space to a basic closed set in the domain space and that settles the continuity of $\mathcal{Z}_S^\chi$. Indeed we shall check that for any $E \in \mathcal{A}$, $(\mathcal{Z}_S^{\chi})^{-1}[S_E]=\mathcal{M}_{\chi_{(X\setminus E)}}$. Towards proving this assertion let $M\in \mathcal{M}_{\chi_{(X \setminus E)}}$, then $\chi_{(X \setminus E)}\in M$. Since $\chi_{(X \setminus E)}$ is invertible on $X\setminus E$, it follows that $E\in \mathcal{Z}_S(\chi_{(X \setminus E)})$ and hence $E\in \mathcal{Z}_S[M]$. This further implies that $\mathcal{Z}_S^\chi[M]\in S_E$ and hence $M\in ({\mathcal{Z}_S^\chi})^{-1} [S_E]$. Thus $\mathcal{M}_{\chi_{(X \setminus E)}} \subseteq ({\mathcal{Z}_S^\chi})^{-1}[S_E$]. Towards proving the reverse inclusion relation. Let $M \in ({\mathcal{Z}_S^\chi})^{-1}[S_E]$. Then $\mathcal{Z}_S^\chi[M]\in S_E$ and therefore $E\in \mathcal{Z}_S[M]$. Consequently $E\in \mathcal{Z}_S(g)$ for some $g\in M$ and surely $g$ is invertible on $X\setminus E$. This implies that, there exist $h\in S(X,\mathcal{A})$ such that $gh|_{(X \setminus E)}=1$. It follows that $\chi_{(X \setminus E)}=gh\chi_{(X\setminus E)}$. Hence $\chi_{(X \setminus E)}\in M$, in other words $M\in \mathcal{M}_{\chi_{(X\setminus E)}}$. Thus $({\mathcal{Z}_S^\chi})^{-1}[S_E] \subseteq \mathcal{M}_{\chi_{(X\setminus E)}}$. Hence $({\mathcal{Z}_S^\chi})^{-1}[S_E]=\mathcal{M}_{\chi_{(X\setminus E)}}$.
\end{proof}
\begin{corollary}\label{Z continuous}
	The map $\mathcal{Z}_S:(\mathcal{M}_S,\tau_S)\rightarrow (\mathcal{U}_X,\tau)$ is a continuous map.

\end{corollary}
 We define an equivalence relation on the set $\mathcal{M}_S $ by the following rule: for $M,N\in \mathcal{M}_S$ we write $M \sim N$ if and only if $\mathcal{Z}_S[M]= \mathcal{Z}_S[N]$. Let $\Pi:\mathcal{M}_S \rightarrow \mathcal{M}/\sim $ be the natural map defined as $\Pi(M)=[M]$, here $[M]$ is the equivalence class in the quotient set $\mathcal{M}_S/\sim$ which contains the element $M\in \mathcal{M}_S$. Endow $\mathcal{M}_S/\sim $, the quotient topology $\tau_q$ on it, induced by the map $\Pi$, thus $\tau_q$
is the largest topology on $\mathcal{M}_S/\sim$, which makes $\Pi$, a continuous map. Define the map $\mathcal{Z}_S^0:\mathcal{M}_S/\sim \rightarrow \mathcal{U}_X $ by the following rule: $\mathcal{Z}_S^0([M])=\mathcal{Z}_S[M]$. Thus $\mathcal{Z}_S^0  \circ \Pi= \mathcal{Z}_S$, in other words, the categorical diagram:
\begin{center}
	\begin{tikzpicture}
		
		\node at (0,0) {$(\mathcal{M}_S,\tau_S)$};
		
		\node at (1.75,0.2) {$\mathcal{Z}_S$};
		\draw[->] (0.75,0) -- (2.6,0);
		\node at (3.2,0) {$(\mathcal{U}_X,\tau)$};
		\draw[->] (0,-0.2) -- (0, -1.5);
		\node at (-0.2,-0.75) {$\Pi$};
		\node at (0,-1.75) {$(\mathcal{M}_S/\sim,\tau_q)$};
		\draw[->] (0.1,-1.5) -- (3.2,-0.25);
		\node at (2,-1) {$\mathcal{Z}_S^0$};
	
	\end{tikzpicture}
\end{center}
is commutative.
\begin{theorem}\label{Z homeomorphism}
	 The map $\mathcal{Z}_S^0:(\mathcal{M}_S/\sim ,\tau_q) \rightarrow (\mathcal{U}_X,\tau)$ is a homeomorphism.
\end{theorem}
\begin{proof}
	 Since by Theorem \ref{preimage of ultra} the map $\mathcal{Z}_S$ is onto $\mathcal{U}_X$, it follows that the map $\mathcal{Z}_S^0$ is onto $\mathcal{U}_X$. The injectivity of the map $\mathcal{Z}_S^0$ follows from the definition of equivalence relation $\sim$ on $\mathcal{M}_S$. To prove the continuity of the map $\mathcal{Z}_S^0$, let $S_E=\{\mathcal{U}\in \mathcal{U}_X:E\in \mathcal{U}\} $ be a typical basic closed set in the space $(\mathcal{U}_X,\tau)$, here $E\in \mathcal{A}
$. Since by the Corollary \ref{Z continuous}, $\mathcal{Z}_S$ is a continuous map, it follows that $\mathcal{Z}_S^{-1}(S_E)$ is a closed set in $(\mathcal{M}_S,\tau_S)$. But $\mathcal{Z}_S^{-1}(S_E)=(\mathcal{Z}_S^0 \circ \Pi)^{-1}(S_E)=\Pi^{-1}({(\mathcal{Z}_S^0)^{-1}}(S_E))$. Hence, since $\Pi:(\mathcal{M}_S,\tau_S)\rightarrow (\mathcal{M}_S/\sim,\tau_q)$ is a quotient map, it follows that $(\mathcal{Z}_S^0)^{-1}(S_E)$ is closed in the quotient space $(\mathcal{M}_S/\sim,\tau_q)$. Thus $\mathcal{Z}_S^0$ turns out to be a continuous map. To complete this theorem we need to show $\mathcal{Z}_S^0$ is a closed map. For that purpose let $K$ be a closed set in $(
\mathcal{M}_S/ \sim,\tau_q)$. But note that $(\mathcal{M}_S,\tau_S)$, like any structure space of a commutative ring with unity, is a compact space and $\Pi$ is a continuous surjection onto $(\mathcal{M}_S/\sim, \tau_q)$. Therefore the later space is also compact. Accordingly $K$ is a compact subspace of $(\mathcal{M}_S/\sim,\tau_q)$. As $\mathcal{Z}_S^0$ is a continuous map, already settled, it follows that $\mathcal{Z}_S^0(K)$ is a compact subset of $(\mathcal{U}_X,\tau)$. Since $(\mathcal{U}_X,\tau)$ is Hausdorff, this implies that $\mathcal{Z}_S^0(K)$ is a closed subset of $(\mathcal{U}_X,\tau)$. Thus it is proved that $\mathcal{Z}_S^0$ is a closed map. 
\end{proof}
The next result offers yet another characterization of Gelfandness of the ring $S(X,\mathcal{A})$.
\begin{theorem}\label{Gelfand to homeomorphism}
	The following three statements are equivalent for a $\chi$-ring $S(X,\mathcal{A})$;
	\begin{enumerate}
		\item $S(X,\mathcal{A})$ is a Gelfand ring.
		\item The map $\mathcal{Z}_S:(\mathcal{M}_S,\tau_S)\rightarrow (\mathcal{U}(X),\tau)$ is a homeomorphism.
		\item $(\mathcal{M}_S,\tau_S)$ is homeomorphic to $(\mathcal{U}(X),\tau)$.
	\end{enumerate}
	
\end{theorem} \begin{proof}
If $S(X,\mathcal{A})$ is a Gelfand ring, then it is clear that the space $(\mathcal{M}_S,\tau_S)$  and $(\mathcal{M}_S/\sim,\tau_q)$ are essentially the same. Hence from the Theorem \ref{Z homeomorphism}, it follows that the map $\mathcal{Z}_S:(\mathcal{M}_S,\tau_S) \rightarrow (\mathcal{U}_X,\tau)$ is a homeomorphism. This proves $(1)$ implies $(2)$.\newline $(2)$ implies $(3)$ is trivial.
\newline If $(\mathcal{M}_S,\tau_S)$ is  homeomorphic to $ (\mathcal{U}_X,\tau)$, then  since $\tau$ is already known to be Hausdorff, it follows that the topology $\tau_S$ on $\mathcal{M}_S$ is also Hausdorff. At this moment we recall that the structure space of a commutative ring $A$ with identity is Hausdorff if and only if $A$ is a Gelfand ring. This is proved in $1.3$ in \cite{MO}. Therefore $S(X,\mathcal{A})$ is a Gelfand ring. This proves $(3)$ implies $(1)$.
\end{proof}

\begin{corollary}
	 The structure space of each intermediate ring of measurable functions, lying between $\mathcal{M}^*(X,\mathcal{A})$ and $\mathcal{M}(X,\mathcal{A})$ is homeomorphic to the space $(\mathcal{U}_X,\tau)$. This follows from Corollary \ref{intermediate corollary} and Theorem \ref{Gelfand to homeomorphism}. This is precisely Theorem $4.8$ in \cite{ABS2020}.
\end{corollary} \begin{corollary}
The $\chi$-ring $S(X,\mathcal{A})$ is Gelfand if and only if $\mathcal{Z}_S^\chi:(\mathcal{M}_S,\tau_\chi)\rightarrow (\mathcal{U}_X,\tau)$ is a homeomorphism. This follows on using Theorems \ref{Gelfand to homeomorphism} and \ref{Z chi continous}.
\end{corollary} 
\begin{remark}
	 If $S(X,\mathcal{A})$ is a Gelfand ring, then all the four topologies $\tau$, $\tau_S$, $\tau_\chi$, $\tau_q$ constructed over appropriate sets are pairwise homeomorphic. 
\end{remark} 
\section{ Conditions for Von-Neumann regularity of the $\chi$-rings. }
 Before offering several  equivalent descriptions of the Von-Neumann regularity of a $\chi$-ring $S(X,\mathcal{A})$ we feel it imperative to initiate a kind of ideal called $\mathcal{Z}_S$-ideal in this ring. 
 \begin{definition}
 	 An ideal $I$ in $S(X,\mathcal{A})$ is called a  $\mathcal{Z}_S$-ideal if for any $f\in S(X,\mathcal{A})$, $\mathcal{Z}_S(f)\subseteq \mathcal{Z}_S[I]$ implies that $f\in I$.
 	  \end{definition}
    We recall in this context the notion of $Z$-ideals in a commutative ring $A$ with unity. An ideal $I$ in $A$ is called a $Z$-ideal if for each $a\in I$, $M_a\subseteq I$, here $M_a$ is the intersection of all the maximal ideals in $A$, which contain $a$. It follows immediately that every maximal ideal in $A$ is a $Z$-ideal. $Z$-ideals in $A$ are introduced in \cite{M1973}. We would like to mention here that inspite of the similarity of notations, the behavior of $\mathcal{Z}_S$-ideals in $S(X,\mathcal{A})$ as opposed to $Z$-ideal, is a bit surprising. The following theorem attests to this fact in clear terms. 
    \begin{theorem}
    	 Each maximal ideal in the $\chi$-ring $S(X,\mathcal{A})$ is a $\mathcal{Z}_S$-ideal if and only if $S(X,\mathcal{A})$ is a Gelfand ring.
    \end{theorem}
\begin{proof}
	First assume that each maximal ideal in $S(X,\mathcal{A})$ is a $\mathcal{Z}_S$-ideal. Let $P$ be a prime ideal in the ring $S(X,\mathcal{A})$ and $M$ and $N$ be two maximal ideals in the ring with $P\subseteq M\cap N$. Then $\mathcal{Z}_S[P] \subseteq \mathcal{Z}_S[M] \cap \mathcal{Z}_S[N]$. It follows from Remark \ref{maxi to ultra} that $\mathcal{Z}_S[P]=\mathcal{Z}_S[N]=\mathcal{Z}_S[M]$= an $\mathcal{A}$-ultrafilter on $X$. Now choosing $f\in M$ arbitrarily, this yields $\mathcal{Z}_S(f)\subseteq \mathcal{Z}_S[M]=\mathcal{Z}_S[N]$. This further implies as $N$ is a $\mathcal{Z}_S$-ideal that $f\in N$. Therefore $M\subseteq N$ and hence $M=N$ because of the maximality of the ideals. Thus it is proved that $S(X,\mathcal{A})$ is a Gelfand ring.
	\newline Conversely, let $S(X,\mathcal{A})$ be a Gelfand ring and $M$ be a maximal ideal in it. Suppose $ f\in S(X,\mathcal{A})$ is such that $\mathcal{Z}_S(f) \subseteq \mathcal{Z}_S[M]$. We shall show that $f\in M$ and that will finish the theorem. It follows from the Theorem \ref{pseudoprime} that $\mathcal{U}=\mathcal{Z}_S[M]$ is an $\mathcal{A}$-ultrafilter. In the proof of the first part of the Theorem \ref{preimage of ultra}, it is further observed that $\mathcal{Z}_S[\mathcal{Z}_\mathcal{M}^{-1}[\mathcal{U}]\cap S(X,\mathcal{A})]=\mathcal{U}$. Let $P=\mathcal{Z}_\mathcal{M}^{-1}[\mathcal{U}]\cap S(X,\mathcal{A})$, then $P$ is a prime ideal in the ring $S(X,\mathcal{A})$. We shall show that $P \subseteq M$. So let $g\in P\setminus\{0\}$, then $X\setminus Z(g) \neq \emptyset $ and hence $g$ is invertible on $X\setminus Z(g)$ in the ring $\mathcal{M}(X,\mathcal{A})$. This implies that $Z(g)\in \mathcal{Z}_\mathcal{M}(g) \subseteq \mathcal{U}$ as $g\in \mathcal{Z}_\mathcal{M}^{-1}[\mathcal{U}]$. Therefore $X\setminus Z(g ) \notin \mathcal{U}$, as $\mathcal{U}$ has finite intersection property. Since $\mathcal{U}=\mathcal{Z}_S[M]$, this implies that $\chi_{Z(g)}\notin M$. But it is an easy verification that $\chi_{Z(g)} \cdot \chi_{X\setminus Z(g)}=0\in M$. Hence it follows that $\chi_{X\setminus Z(g)}\in M$. But then $g=\chi_{X\setminus Z(g)} \cdot g \in M$. Thus it is proved that $P \subseteq M$. To complete the theorem we need to show that $<P,f>$ is a proper ideal in $S(X,\mathcal{A})$ for this will imply in-view of assumed Gelfandness of the ring $S(X,\mathcal{A})$ that $f\in M$, which is exactly what we need to check. Indeed by Theorem \ref{pseudoprime}, $\mathcal{Z}_S[P]$ is an $\mathcal{A}$-ultrafilter on $X$ and $\mathcal{Z}_S[P]\subseteq \mathcal{Z}_S[M]=\mathcal{U}$ implies that $\mathcal{Z}_S[P]=\mathcal{Z}_S[M]=\mathcal{U}$. Therefore $\mathcal{Z}_S(f)\subseteq \mathcal{Z}_S[M]$ $\implies$ $\mathcal{Z}_S(f)\subseteq \mathcal{Z}_S[P]$. Consequently $\mathcal{Z}_S[<P,f>]=\mathcal{Z}_S[P]=\mathcal{U}$. This implies that $<P,f>$ is a proper ideal in $S(X,\mathcal{A})$ for $\emptyset\in \mathcal{Z}_S[S(X,\mathcal{A})]$ because $\emptyset\in \mathcal{Z}_S(1)$.
\end{proof}
The following elementary but useful fact will be employed to prove our main theorem in this section, offering various equivalent formulations for the Von-Neumann regularity of the ring $S(X,\mathcal{A})$.
\begin{theorem}
	\label{intersection of prime ideals}
	Any $\mathcal{Z}_S$-ideal $I$ in $S(X,\mathcal{A})$ is the intersection of all prime ideals, which contain it.
\end{theorem} 
\begin{proof}
	Let $\sqrt{I}$ be the intersection of all prime ideals in $S(X,\mathcal{A})$ which contain $I$. Suppose $f\in \sqrt{I}$, then there exists $n\in \mathbb{N}$ such that $f^n\in I$. This implies that $\mathcal{Z}_S(f^n)\subseteq \mathcal{Z}_S[I]$. But $\mathcal{Z}_S(f^n)=\mathcal{Z}_S(f)$, therefore $\mathcal{Z}_S(f)\subseteq \mathcal{Z}_S[I]$. As $I$ is a $\mathcal{Z}_S$-ideal, this implies $f\in I$. Thus $I=\sqrt{I}$.
\end{proof}
\begin{theorem}
	The following statements are equivalent for a $\chi$-ring $S(X,\mathcal{A})$.\begin{enumerate}
		\item Each prime ideal in $S(X,\mathcal{A})$ is maximal.
		\item $S(X,\mathcal{A})$ is a Von-Neumann regular ring in the sense given $f\in S(X,\mathcal{A})$ there exists $g\in S(X,\mathcal{A})$ such that $f=f^2g$.
		\item For each $f\in S(X,\mathcal{A})$, $Z(f)\in \mathcal{Z}_S(f)$.
		\item  For each ideal $I$ in $S(X,\mathcal{A})$, $Z[I]=\{Z(f):f\in I\}$ is an $\mathcal{A}$-filter on $X$,
	
		 and $Z^{-1}[Z[I]]=I$. \item The map $Z:\mathcal{I}_S\rightarrow \mathcal{A}_X$ given by $I\rightarrow Z[I]$ is a bijection on the set $\mathcal{I}_S$ of all ideals in $S(X,\mathcal{A})$ onto the set $\mathcal{A}_X$ of all $\mathcal{A}$-filters on $X$.
		\item For all $f,g\in S(X,\mathcal{A})$, $<f,g>=<f^2+g^2>$.
		\item If $P$ is a prime ideal in $S(X,\mathcal{A})$ and $f\in P$ then $Z(f)\in \mathcal{Z}_S[P]$.\item If $M$ is a maximal ideal in $S(X,\mathcal{A})$ and $f\in M$ then $Z(f)\in \mathcal{Z}_S[M]$.
		\item If $Z(f)\subseteq Z(g)$ where $f,g \in S(X,\mathcal{A})$, then $g$ is a multiple of $f$ in the ring $S(X,\mathcal{A})$.
		\item  Every ideal in $S(X,\mathcal{A})$ is a $\mathcal{Z}_S$-ideal.
		\item Every principal ideal in $S(X,\mathcal{A})$ is generated by an idempotent.
		\item Every ideal in $S(X,\mathcal{A})$ is the intersection of all prime ideals which contain it.
		\item Every ideal in $S(X,\mathcal{A})$ is the intersection of all maximal ideals which contain it.
		\item  For any ideal $I$ in $S(X,\mathcal{A})$, $I=\mathcal{Z}_S^{-1}[\mathcal{Z}_S[I]]$, here $\mathcal{Z}_S^{-1}[\mathcal{Z}_S[I]]$ stands for $\mathcal{Z}_\mathcal{M}^{-1}[\mathcal{Z}_S[I]]\cap S(X,\mathcal{A})$.
		\item Each ideal in $S(X,\mathcal{A})$ is a $Z$-ideal.
	\end{enumerate}
	
\end{theorem}
\begin{proof}
	$(1)\iff (2)$ : It follows from a result proved in \cite{G}  which says that a commutative reduced ring $A$ with identity is Von-Neumann regular if and only if every prime ideal in $A$ is maximal.
	\newline $(2) \implies (3)$ : Let $f\in S(X,\mathcal{A})$. Because of $(2)$, there exists $h\in S(X,\mathcal{A})$ such that $f=f^2h$, hence $f(1-fh)=0$. Thus $f$ is invertible on $X\setminus Z(f)$, therefore $Z(f)\in \mathcal{Z}_S(f)$.\newline $(3)\implies (2)$ : Let $f\in S(X,\mathcal{A})$. Then $Z(f)\in \mathcal{Z}_S(f)$ implies that there exists $h\in S(X,\mathcal{A})$ such that $fh|_{(X\setminus Z(f))}=1$. This shows that $f=f^2h$. Hence $(2)$ holds.\newline $(2)\implies (4)$ : Suppose $(2)$ holds. Then it is clear that an $f\in S(X,\mathcal{A})$ is a unit in this ring if and only if $Z(f)=\emptyset$. Consequently for an ideal (proper) $I$ in $S(X,\mathcal{A})$, $Z[I]=\{Z(f):f\in I\}$ is an $\mathcal{A}$-filter on $X$. Furthermore if $f\in S(X,\mathcal{A})$ and $g\in I$ are such that $Z(f)=Z(g)$, then since $g=g^2h$ for some $h\in S(X,\mathcal{A})$, this implies that $gh=1$ on $X\setminus Z(f)=X\setminus Z(g)$, consequently $f=fgh\in I$. This proves that $Z^{-1}[Z[I]]=I$.
	\newline $(4)\iff (5)$ : Trivial. \newline $(2) \implies(6)$ : Let $(2)$ hold and $f,g\in S(X,\mathcal{A})$. This implies that $<f^2+g^2> \subseteq <f,g>$. To prove the reverse inclusion relation, we observe due to the assumed Von-Neumann regularity of $S(X,\mathcal{A})$ that there exists $h\in S(X,\mathcal{A})$ such that $(f^2+g^2)h=1$ on the set $X\setminus Z(f^2+g^2)=(X\setminus Z(f)\cup (X\setminus Z(
	g))$. This clearly implies that $f=fh(f^2+g^2)$ on $X\setminus Z(f)$ and consequently $f=fh(f^2+g^2)$, thus $f\in <f^2+g^2>$. Analogously $g\in <f^2+g^2>$, hence $<f,g>=<f^2+g^2>$.\newline $(6) \implies (2)$ : Trivial.
	\newline $(4)\implies (2)$ : Let $(4)$ be true. Let $f\in S(X,\mathcal{A})$. If $f$ is a unit in $S(X,\mathcal{A})$, then $fg=1$ for some $g\in S(X,\mathcal{A})$, this implies $f=f^2g$. Assume therefore that $f$ is not a unit in $S(X,\mathcal{A})$. Then the principal ideal $<f^2>$ is a (proper) ideal in $S(X,\mathcal{A})$. Furthermore $Z(f)=Z(f^2)$ implies by virtue of the second part of the assumed condition $(4)$ that $f\in <f^2>$. Thus it is proved that $S(X,\mathcal{A})$ is Von-Neumann regular.
	\newline
	 $(3) \implies (7)$ : Trivial. \newline $(7)\implies (8)$ : Trivial.\newline $(8) \implies (1)$ : Let $(8)$ be true and $P$ be a prime ideal in $S(X,\mathcal{A})$. There exists a maximal ideal $M$ in $S(X,\mathcal{A})$ such that $P \subseteq M$. We shall show that $P=M$. So let $f\in M$, then by the assumed condition $(8)$, $Z(f)\in \mathcal{Z}_S[M]$. But by the Remark \ref{maxi to ultra} $\mathcal{Z}_S[M]=\mathcal{Z}_S[P]$, hence $Z(f)\in \mathcal{Z}_S[P]$. Since $\mathcal{Z}_S[P]$ is an $\mathcal{A}$-filter on $X$, it follows that $X\setminus Z(f)\notin\mathcal{Z}_S[P]$. This implies that $\chi_{Z(f)}\notin P$. But then $\chi_{(X\setminus Z(f))}\cdot \chi_{Z(f)}=0$, implies that $\chi_{(X\setminus Z(f))}\in P$. But $f=f\cdot \chi_{(X\setminus Z(f))}\in P$. Hence $P=M$.\newline $(2)\implies (9)$ : Let $(2)$ hold good and $Z(f)\subseteq Z(g)$, $f,g\in S(X,\mathcal{A})$. By $(2)$, there exists $h\in S(X,\mathcal{A})$ such that $f=f^2h$. Consequently $g=fgh,$ thus $g$ is a multiple of $f$.
	 \newline $(9)\implies (2)$ : Let $(9)$ be true. Choose $f\in S(X,\mathcal{A})$. Then $Z(f)=Z(f^2)$ implies in view of $(9)$ that $f$ is a multiple of $f^2$ in $S(X,\mathcal{A})$. Hence $S(X,\mathcal{A})$ is Von-Neumann regular.
	 \newline $(3)\implies (10)$: Let $(3)$ be true and $I$ be an ideal in $S(X,\mathcal{A})$. Let $f\in S(X,\mathcal{A})$ be such that $\mathcal{Z}_S(f)\subseteq \mathcal{Z}_S[I]$. To show that $f\in I$. Now by $(3)$, $Z(f)\in \mathcal{Z}_S(f)$. Hence $Z(f)\in \mathcal{Z}_S[I]$. This means that $Z(f)\in \mathcal{Z}_S(g)$ for some $g \in I$. Therefore there exists $k\in S(X,\mathcal{A})$ such that $gk=1$ on $X\setminus Z(f)$. This implies that $f=fgk$, hence $f\in I$ as $g\in I$.\newline $(10) \implies(2)$: Let $(10)$ be true. Suppose $f\in S(X,\mathcal{A})$, then $\mathcal{Z}_S(f)=\mathcal{Z}_S(f^2)\subseteq \mathcal{Z}_S[<f^2>]$, here $<f^2>
	 $ is the principal ideal in $S(X,\mathcal{A})$ generated by $f^2$. By $(10)$, $<f^2>$ is a $\mathcal{Z}$-ideal in $S(X,\mathcal{A})$, therefore $f\in <f^2>$ and hence $(2)$ follows.
	 \newline $(2) \implies (11)$ : Let $(2)$ be true and $f\in S(X,\mathcal{A})$. Then $f=f^2h$ for some $h\in S(X,\mathcal{A})$. Since $f^2h^2=fh$, it is clear that $fh$ is an idempotent element of $S(X,\mathcal{A})$. It is easy to realize that $<f>=<fh>$.\newline $(11) \implies (2)$ : Suppose $(11)$ is true. Choose $f\in S(X,\mathcal{A})$, then there is an idempotent element $e\in S(X,\mathcal{A})$ such that $<f>=<e>$. Therefore there exists $k\in S(X,\mathcal{A})$ such that $f=ke=ke^2$. On the other hand $e=lf$ for some $l\in S(X,\mathcal{A})$. This implies that $f=kl^2f^2$. Thus $(2)$ is true.
	 \newline $(10) \implies (12)$ : Follows from the Theorem $\ref{intersection of prime ideals}$.
	 \newline $(12)\implies (2)$ : Let $(12)$ be true and $f\in S(X,\mathcal{A})$. Then $<f^2>$ is the intersection of all prime ideals in $S(X,\mathcal{A})$ which contain $f^2$. But as $f$ and $f^2$ belong to exactly the same set of prime ideals in $S(X,\mathcal{A})$, it follows that $f\in <f^2>$.\newline $(12) \implies (13)$ : Follows from $(12)\implies (2)$ and $(2) \iff (1)$. \newline $(13) \implies (12)$ : Is immediate because each maximal ideal in $S(X,\mathcal{A})$ is prime.\newline $(2)\implies (14)$ : Let $(2)$ be valid and $I$ be an ideal in $S(X,\mathcal{A})$. Now it is not difficult to show that $\mathcal{Z}_\mathcal{M}^{-1}(\mathcal{Z}_S[I])\cap S(X,\mathcal{A}) \subseteq I$. Indeed if $f\in \mathcal{Z}_\mathcal{M}^{-1}(\mathcal{Z}_S[I])\cap S(X,\mathcal{A})$, then $\mathcal{Z}_\mathcal{M}(f)\subseteq\mathcal{Z}_S[I]$. But since $\mathcal{Z}_S[I]$ is an $\mathcal{A}$-filter on $X$ and $\mathcal{Z}_\mathcal{M}(f)$ contains all these $\mathcal{A}$-measurable sets which contain $Z(f)$( the observation just proceeding the statement of Theorem \ref{ preimage of filter} ), it follows that $Z(f) \in \mathcal{Z}_S[I]$. But this means that $Z(f)\in \mathcal{Z}_S(g)$ for some $g\in I$. Thus $g$ is invertible on $X\setminus Z(f)$ meaning that there exists $h\in S(X,\mathcal{A})$ such that $gh=1$ on $X\setminus Z(f)$. This implies that $f=fgh$ on $X$ and hence $f\in I$. Thus it is proved that $\mathcal{Z}_\mathcal{M}^{-1}(\mathcal{Z}_S[I])\cap S(X,\mathcal{A}) \subseteq I$. To prove the reverse inclusion relation let $f\in I$. Then by $(3)$, $Z(f)\in \mathcal{Z}_S(f)$ and hence $Z(f)\in \mathcal{Z}_S[I]$. Since $\mathcal{Z}_\mathcal{M}(f)$ contains all these $\mathcal{A}$-measurable sets which contain $Z(f)$( already mentioned above ) and $\mathcal{Z}_S[I]$ is an $\mathcal{A}$-filter on $X$, containing $Z(f)$, it follows that $\mathcal{Z}_\mathcal{M}(f)\subseteq \mathcal{Z}_S[I]$. This implies that $f\in \mathcal{Z}_\mathcal{M}^{-1}(\mathcal{Z}_S[I])\cap S(X,\mathcal{A})$. So it is proved that $I=\mathcal{Z}_\mathcal{M}^{-1}(\mathcal{Z}_S[I])\cap S(X,\mathcal{A})$.\newline $(14) \implies (1)$ : Let $(14)$ be true. Suppose $P$ is a prime ideal in $S(X,\mathcal{A})$ and $M$ a maximal ideal in the same ring with $P \subseteq M$. It follows from the Remark \ref{maxi to ultra} that $\mathcal{Z}_S[P]=\mathcal{Z}_S[M]$. This in conjunction with the hypothesis $(14)$ implies $P=M$.\newline $(2)\iff (15)$ : Follows from the Theorem $1.2$ in \cite{M1973}.
	 \newline The theorem is completely proved. 
\end{proof}

\section{$u_\mu$-topology and $m_\mu $-topology on the ring $S(X,\mathcal{A})$}
In this section we make an additional assumption on the ring $S(X,\mathcal{A})$:  
each constant function $r$, $r\in \mathbb{R}$ belongs to $S(X,\mathcal{A})$. We note that each intermediate ring of real valued measurable functions satisfies this additional hypothesis.
It is easy to construct example of a $\chi$-ring $S(X,\mathcal{A})$ which contains the constant functions $r$, $r\in \mathbb{R}$ and which is not an intermediate ring lying between $\mathcal{M}^*(X,\mathcal{A})$ and $\mathcal{M}(X,\mathcal{A})$. The following pair of results can be proved by using some routine computations. We assume that $\mu:\mathcal{A}\rightarrow[0,\infty]$ is a measure. \begin{theorem}
	 The family $\{u(f,\epsilon,E_0):f\in S(X,\mathcal{A}), \epsilon>0, E_0\in \mathcal{A}\}$ is an open base for a topology, which we call $u_\mu$-topology on $S(X,\mathcal{A})$. Furthermore $S(X,\mathcal{A})$ with this $u_\mu$-topology is a topological group.
\end{theorem}
\begin{theorem}
	\label{topological group}
	The family $\{m(f,u,E_0):f\in S(X,\mathcal{A}) ,u\in \mathcal{M}(X,\mathcal{A}), u(x)>0 \text{ for all }x\in X,  E_0\in \mathcal{A}_0\}$ is an open base for a topology, which we designate as $m_\mu$-topology on $S(X,\mathcal{A})$. $S(X,\mathcal{A})$ with the $m_\mu$-topology is a topological ring.
\end{theorem}
 We would like to mention at this point that the above two theorems still hold good for an arbitrary $\chi$-ring $S(X,\mathcal{A})$ 
 without the hypothesis that $r\in S(X,\mathcal{A})$ for each $r\in \mathbb{R}$.

 In this section the ring $L^\infty(X,\mathcal{A},\mu)$ of all essentially bounded (with respect to the measure $\mu$) measurable functions over $(X,\mathcal{A})$ will appear from time to time. We recall that $L^\infty(X,\mathcal{A},\mu)=\{f\in \mathcal{M}(X,\mathcal{A}): f $ is essentially bounded over $  X $ $\text{ in the sense that there exists }\lambda>0 \text{ and an } E\in \mathcal{A}_0 \text{ such that} $ $|f(x)|\leq \lambda \text{ for all } x\in X\setminus E\}$. It is clear that $L^\infty(X,\mathcal{A},\mu)$ is an intermediate ring of measurable functions over $(X,\mathcal{A})$. The following theorem determines the component of $0$ in a $\chi$-ring, in the $u_\mu $ topology.
\begin{theorem}\label{component of 0 }
	Let $S(X,\mathcal{A})$ be a $\chi$-ring. Then  the component of $0$ in $S(X,\mathcal{A})$ in the $u_\mu$-topology is $S(X,\mathcal{A}) \cap L^\infty(X,\mathcal{A},\mu )$.
\end{theorem}
\begin{proof}
	It is  easy to check that for any $f\in L^\infty(X,\mathcal{A},\mu)$, the set $\{g\in \mathcal{M}(X,\mathcal{A},\mu):|f(x)-g(x)|\leq 1 \text{ almost everywhere }(\mu) \text{ on }X|\}\subseteq L^\infty(X,\mathcal{A},\mu)$. On the other hand if $f\notin L^\infty(X,\mathcal{A},\mu)$, $ f\in \mathcal{M}(X,\mathcal{A},\mu) $ then $\{h\in \mathcal{M}(X,\mathcal{A},\mu):|f(x)-h(x)|\leq 1 \text{ almost everywhere }(\mu) \text{ on } X\}\cap L^\infty(X,\mathcal{A},\mu)=\emptyset$. These two elementary relations imply that $L^\infty(X,\mathcal{A},\mu)$ is a clopen subset of $\mathcal{M}(X,\mathcal{A},\mu)$ if the $u_\mu$-topology is imposed on this later set. Since the $u_\mu$-topology already introduced on $S(X,\mathcal{A})$ is the relative topology on $S(X,\mathcal{A})$ inherited from the $u_\mu$-topology on the whole of $\mathcal{M}(X,\mathcal{A})$, it follows that $S(X,\mathcal{A})\cap L^\infty(X,\mathcal{A},\mu)$ is a clopen set in the space $S(X,\mathcal{A})$ equipped with the $u_\mu$-topology and of course $0\in S(X,\mathcal{A})\cap L^\infty(X,\mathcal{A},\mu)$. To complete the theorem we need to prove only that $S(X,\mathcal{A})\cap L^\infty(X,\mathcal{A},\mu)$ is a connected subset of the space $S(X,\mathcal{A})$ in the $u_\mu$-topology. For that purpose we employ the function $\phi_f:\mathbb{R}\rightarrow S(X,\mathcal{A})$ defined by $\phi_f(r)=r\cdot f,r\in \mathbb{R}$ and for each $f\in S(X,\mathcal{A})$, akin to the function introduced in  \cite{AMM2012}. Since for $f\in S(X,\mathcal{A})$, $r,s\in \mathbb{R}$, $|\phi_f(r)-\phi_f(s)|=|f||r-s|$, it is clear that $\phi_f$ is a continuous map on $\mathbb{R}$ if and only if $f\in S(X,\mathcal{A})\cap L^\infty(X,\mathcal{A},\mu)$. Hence for any $f\in S(X,\mathcal{A})\cap L^\infty(X,\mathcal{A},\mu)$, $\phi_f(\mathbb{R})$ is a connected subsets of $S(X,\mathcal{A})\cap L^\infty(X,\mathcal{A},\mu)$ and therefore $S(X,\mathcal{A})\cap L^\infty(X,\mathcal{A},\mu)= \bigcup\limits_{f\in S(X,\mathcal{A})\cap L^\infty(X,\mathcal{A},\mu)}\phi_f(\mathbb{R})$=the union of a family of connected subset of $S(X,\mathcal{A})$, each of which contain $0$ $\equiv$ a connected subset of $S(X,\mathcal{A})$ containing $0$.
	
\end{proof}
\begin{corollary}\label{Contains}
	A $\chi$-ring $S(X,\mathcal{A})$  is connected in the $u_\mu$-topology if and only if $S(X,\mathcal{A}) \subseteq L^\infty(X,\mathcal{A},\mu)$. 
\end{corollary}
\begin{proof}
	 If  $S(X,\mathcal{A})\subseteq L^\infty(X,\mathcal{A},\mu)$, then $S(X,\mathcal{A})=S(X,\mathcal{A})\cap L^\infty(X,\mathcal{A},\mu)$, which is connected in the $u_\mu$-topology by Theorem \ref{component of 0 }. On the other hand if $S(X,\mathcal{A}) \not\subset L^\infty(X,\mathcal{A},\mu)$, then $S(X,\mathcal{A})\cap L^\infty(X,\mathcal{A},\mu)$ is a non empty clopen subset of the space $S(X,\mathcal{A})$ with $u_\mu$-topology (see the proof of the Theorem \ref{component of 0 }) and is a proper subset of $S(X,\mathcal{A})$. Consequently $S(X,\mathcal{A})$ becomes disconnected in the $u_\mu$-topology.
\end{proof}
 Before examining the problem, when do the $u_\mu$-topology and $m_\mu$-topology on $S(X,\mathcal{A})$ coincide, we first make the following useful observation regarding the pseudometrizability of $S_u(X,\mathcal{A},\mu)$, we will let $S_u(X,\mathcal{A},\mu)$ stand for the ring $S(X,\mathcal{A})$ equipped with the $u_\mu$-topology. We will denote by $S_m(X,\mathcal{A},\mu)$, the ring $S(X,\mathcal{A})$ with the $m_\mu$-topology.
 \begin{theorem}\label{pseudometrizable}
 For $f,g\in S(X,\mathcal{A})$, we set $$d(f,g)=\min\{1, \inf\limits_{A\in \mathcal{A}_0}  \sup\limits_{x\in X\setminus A}|f(x)-g(x)|\}$$  Then $d$ defines a pseudometric on $S(X,\mathcal{A})$. Furthermore the topology induced by $d$ is identical to the $u_\mu$-topology on $S(X,\mathcal{A})$.
 \end{theorem} We omit the routine proof of this theorem.

\begin{theorem}
	 The following statements are equivalent:\begin{enumerate}
	 	\item $S_u(X,\mathcal{A},\mu)=S_m(X,\mathcal{A},\mu)$.
	 	\item $S_m(X,\mathcal{A},\mu)$ is pseudometizable.
	 	\item $S_m(X,\mathcal{A},\mu)$ is first countable.
	 	\item There exist at most finitely many pairwise disjoint $\mathcal{A}$-measurable sets each with a positive $\mu$-measure.
	 	\item $\mathcal{M}(X,\mathcal{A})= L^\infty(X,\mathcal{A},\mu)$.
	 
	 \end{enumerate}
\end{theorem}
\begin{proof}
	$(1)\implies(2)$ :  Follows from Theorem \ref{pseudometrizable}.
	\newline Since every pseudometrizable space is first countable, $(2)\implies(3)$ follows immediately. \newline $(3)\implies (4)$ : Let $(3)$ be true. If possible let $(4)$ be false. Then there exists a countably infinite family of pairwise disjoint measurable sets $\{A_n\}_{n=1}^\infty$ in $X$ such that $\mu(A_n)>0$ for each $n\in \mathbb{N}$. Now from $(3)$, there exists a countable local open base $\{m_\mu(0,u_n,E_n): u_n(x)>0 \text{ for all }x\in X, u_n\in \mathcal{M}(X,\mathcal{A}), E_n\in \mathcal{A}_0 \}$ about the point $0$ in $S_m(X,\mathcal{A},\mu)$. Define the function $u:X\rightarrow \mathbb{R}$ as follows: $u(x)=\begin{cases}
		\frac{1}{2}u_n(x) \text{ if }  x\in A_n  \text{ for some } n\in \mathbb{N}\\ 1\text{ if } x\in X\setminus (\bigcup\limits_{n=1}^\infty A_n)
	\end{cases}$. Then $u\in \mathcal{M}(X,\mathcal{A})$ and $u(x)>0$ for all $x\in X$. But it is easy to realize that there does not exist any $n\in \mathbb{N}$ for which $m(0,u_n)  \subseteq m(0,u)$.  This contradicts the first countablity of $S_m(X,\mathcal{A},\mu)$.\newline $(4) \implies (5)$ : Choose $f\in \mathcal{M}(X,\mathcal{A})$, $f\geq 0$ on $X$. Then for each $n\in \mathbb{N}$, the set $F_n=\{x\in X:n<f(x)\leq n+1\}$ is a measurable set. Therefore $\{F_n:n\in \mathbb{N}\}$ is a countable family of pairwise disjoint measurable sets. The hypothesis $(4)$ implies that there exist $k\in \mathbb{N}$ such that $\mu(F_n)=0$ for all $n\geq k$. Consequently $f(x)\leq k+1$ almost everywhere on $X(\mu)$. Hence $f\in L^\infty(X,\mathcal{A},\mu )$.
\newline $(5)\implies (1)$ : Let $(5)$ hold. Since the $m_\mu$-topology on $S(X,\mathcal{A})$ is finer than the $u_\mu$-topology, it is enough to verify for an arbitrary $f\in S(X,\mathcal{A})$, $u\in \mathcal{M}(X,\mathcal{A})$, $u(x)>0$ for all $x\in X$ and an $E_0\in \mathcal{A}_0$ that $m(f,u,E_0)$ is an open set in the $u_\mu$-topology. Indeed the  assumed condition $(5)$ implies that there exists $\lambda>0$ in $ \mathbb{R}$ such that $u(x)\geq\lambda$ for each $x\in X\setminus F_0$ for some $F_0\in \mathcal{A}_0$. This further implies that $u(f,\lambda,E_0\cap F_0)\subseteq m(f,u,E_0)$ and $E_0\cap F_0\in \mathcal{A}_0$.

\end{proof}

  \bibliographystyle{plain}

\end{document}